\documentclass[a4paper,reqno,11pt]{amsart}
\usepackage[includehead,includefoot,margin=25mm]{geometry}
\usepackage{times}
\usepackage[english]{babel}
\usepackage{amsmath}
\usepackage{amssymb}
\usepackage{amscd}
\usepackage{amsthm}
\usepackage{euscript}
\usepackage[all]{xy}
\usepackage{amsfonts,amsbsy,amssymb,amsmath,amsthm}

\usepackage{enumerate}
\usepackage{color}

\newcommand{\homo}[0]{\ensuremath{\operatorname{Hom}}}
\newcommand{\GA}[0]{\ensuremath{\mathbb{G}_{\mathrm{a}}}}
\newcommand{\GM}[0]{\ensuremath{\mathbb{G}_{\mathrm{m}}}}
\newcommand{\OO}[0]{\ensuremath{\mathcal{O}}}
\newcommand{\Loc}[0]{\ensuremath{\operatorname{Loc}}}
\newcommand{\tail}[0]{\ensuremath{\operatorname{tail}}}
\newcommand{\rank}[0]{\ensuremath{\operatorname{rank}}}
\newcommand{\SF}[0]{\ensuremath{\mathcal{S}}}

\newcommand{\supp}[0]{\ensuremath{\operatorname{Supp}}}
\newcommand{\CF}[0]{\ensuremath{\mathcal{C}}}
\newcommand{\KK}[0]{\ensuremath{\mathbf{k}}}

\newcommand{\ord}[0]{\ensuremath{\operatorname{ord}}}
\newcommand{\face}[0]{\ensuremath{\operatorname{face}}}

\newcommand{\partiala}[0]{\ensuremath{D}}

\newtheorem{theorem}{Theorem}[section]
\newtheorem*{theorem*}{Theorem}

\newtheorem{lemma}[theorem]{Lemma}
\newtheorem{proposition}[theorem]{Proposition}

{\theoremstyle{remark}
  \newtheorem{remark}[theorem]{Remark}}
{\theoremstyle{definition}
  \newtheorem{definition}[theorem]{Definition}

  \newtheorem{example}[theorem]{Example}

}

\DeclareMathOperator{\Spec}{Spec}
\DeclareMathOperator{\Supp}{Supp}
\DeclareMathOperator{\Hom}{Hom}
\DeclareMathOperator{\Aut}{Aut}

\DeclareMathOperator{\pol}{Pol}

\DeclareMathOperator{\cone}{cone}

\DeclareMathOperator{\SL}{SL}
\DeclareMathOperator{\Lie}{Lie}
\DeclareMathOperator{\LIE}{Lie_{K_0}}
\DeclareMathOperator{\Der}{Der}

\DeclareMathOperator{\relint}{rel.int}

\DeclareMathOperator{\did}{div}

\def\TT{{\mathbb T}}
\def\ZZ{{\mathbb Z}}

\def\QQ{{\mathbb Q}}

\def\DD{{\mathbb D}}

\def\DDD{\mathcal{D}}

\def\sl{\mathfrak{sl}}
\def\gf{\mathfrak{g}}
\def\Nf{\mathfrak{N}}
\def\nf{\mathfrak{n}}

\begin{document}
\date{} \title[ Lie algebras of vertical derivations on
T-varieties]{Lie algebras of vertical derivations on \\ semiaffine
  varieties with torus actions}

\author{Ivan Arzhantsev} %
\thanks{The first author was supported by the grant RSF 19-11-00172} %
\address{National Research University Higher School of Economics,
  Faculty of Computer Science, Pokrovsky Boulevard 11, Moscow, 109028
  Russia} %
\email{arjantsev@hse.ru}

\author{Alvaro Liendo} %
\thanks{The second author was partially supported by the projects
  Fondecyt Regular 1160864 and 1200502} %
\address{Instituto de Matem\'atica y F\'isica, Universidad de Talca,
  Casilla 721, Talca, Chile} %
\email{aliendo@inst-mat.utalca.cl}

\author{Taras Stasyuk} %
%\thanks{}
\address{Department of Mathematics and Mechanics, Lomonosov Moscow
State University, Leninskie Gory~1, Moscow, 119991 Russia}%
\email{taras4834@gmail.com}

\subjclass[2010]{Primary 13N15, 14L30; \ Secondary 14M25, 17B66}

\keywords{Torus action, divisorial fan, additive group action, Demazure 
root, derivation, Lie algebra}

\maketitle

\begin{abstract}
  Let $X$ be a normal variety endowed with an algebraic torus
  action. An additive group action $\alpha$ on $X$ is called vertical
  if a general orbit of $\alpha$ is contained in the closure of an
  orbit of the torus action and the image of the torus normalizes the
  image of $\alpha$ in $\Aut(X)$. Our first result in this paper is a
  classification of vertical additive group actions on $X$ under the
  assumption that $X$ is proper over an affine variety. Then we
  establish a criterion as to when the infinitesimal generators of a
  finite collection of additive group actions on $X$ generate a
  finite-dimensional Lie algebra inside the Lie algebra of derivations
  of $X$.
\end{abstract}

\section{Introduction}

Let $\KK$ be a field of characteristic zero.  The $n$-dimensional
algebraic torus $\TT$ over $\KK$ is the algebraic group $\GM^n$, where
$\GM$ is the multiplicative group, i.e., the set $\KK^*$ of non-zero
elements in the base field endowed with its natural structure of
algebraic group under multiplication. A $\TT$-variety is a normal
variety endowed with a faithful action of the algebraic torus. The
complexity of a $\TT$-variety is the codimension of a general orbit
and since the action is faithful the complexity equals
$\dim X-\dim \TT$. There are well known combinatorial descriptions of
$\TT$-varieties. The $\TT$-varieties of complexity zero are called
toric varieties and were first introduced by Demazure in \cite{De},
see the textbooks \cite{Oda,Fu,CLS} for a modern account on the
subject. They are described by certain collections of rational
polyhedral cones called fans. For higher complexity,
several partial classifications were given until a full classification
of $\TT$-varieties was achieved in \cite{AlHa06,AHS08}, see
\cite{AIPSV} and the references therein for a historical account.

Let now $\GA$ be the additive group, i.e., the base field $\KK$
endowed with its natural structure of algebraic group under
addition. Let $X$ be a variety and $\alpha:\GA\times X\rightarrow X$ be a
$\GA$-action on $X$. The \emph{infinitesimal generator} of the action is the
derivation $\partiala:\OO_X\rightarrow \OO_X$ of the structure sheaf
of $X$ given
by
$$\Gamma(U,\partiala):\Gamma(U,\OO_X)\rightarrow\Gamma(U,\OO_X),\qquad
f\mapsto \left[\frac{d}{dt}(\alpha^*f)\right]_{t=0}\,.$$ An additive
group action on a $\TT$-variety $X$ is called \emph{compatible} if the
image of $\TT$ in $\Aut(X)$ is contained in the normalizer of the
image of the subgroup $\GA$. A compatible $\GA$-action is called \emph{vertical}
if, moreover, a general $\GA$-orbit is contained in the closure of a
$\TT$-orbit. Letting $\KK(X)$ be the field of rational functions on
$X$, we have that a compatible $\GA$-action is vertical if and only if
$\KK(X)^{\TT}\subseteq \KK(X)^{\GA}$. In \cite{De} a description of
vertical $\GA$-actions on toric varieties was given and in
\cite{Lie10,Li} vertical $\GA$-actions on affine
$\TT$-varieties were described. Recall that an algebraic variety $X$ is
called semiaffine if the morphism $X\rightarrow \Spec \Gamma(X,\OO_X)$
induced by the global sections functor is proper \cite{GoLa73}. In
particular, affine and complete varieties are
semiaffine. Our first main result contained in
Theorem~\ref{ver-regular} is a generalization of such descriptions
to vertical $\GA$-actions on semiaffine varieties. Our
classification is given in terms of the combinatorial description of
$\TT$-varieties in \cite{AlHa06,AHS08}. The infinitesimal generators
of vertical $\GA$-actions are in correspondence with certain triples
$(\phi,\lambda_\rho,\chi^e)$ where $\phi$ is a rational function in
$\KK(X)^\TT$, $\lambda_\rho$ is a $1$-parameter subgroup of $\TT$
corresponding to a ray of the fan $\Sigma$ of the normalization of a
general $\TT$-orbit closure and $\chi^e$ is a character of $\TT$ such
that $\langle\rho,e\rangle=-1$ where $\langle\cdot,\cdot\rangle$ 
is the canonical pairing realizing
the $1$-parameter subgroup lattice of a torus as the dual of the
character lattice. We denote the infinitesimal generator of the
$\GA$-action by $D_{\phi,\rho,e}$ and we also call them \emph{root
derivations} since the image of the~$\GA$-action in $\Aut(X)$ is 
a~root subgroup.

Let now $\DD=\{D_{\phi_i,\rho_i,e_i}\}_{i=1}^m$ be a finite set of
root derivations.  The set $\DD$ is called \emph{cyclic} if
${\langle \rho_{i+1}, e_i\rangle>0}$ for all $i=1,\ldots,m$, where we
set $\rho_{m+1}:=\rho_1$. Furthermore, a set $\DD$ is called
\emph{simple} if $\prod_{i=1}^m\phi_i\in\KK$ and  
$\langle \rho_i,e_i\rangle=-1$, $\langle \rho_{i+1},e_i\rangle=1$ and
$\langle \rho,e_i\rangle=0$ for all $i$ and for all rays $\rho$ of the fan $\Sigma$
different from $\rho_i$ and $\rho_{i+1}$. Our second main result in
this paper is the following theorem, see Theorem~\ref{findim} for a
more precise statement.

\begin{theorem*}
  Let $\DD=\{D_{\phi_i,\rho_i,e_i}\}_{i=1}^m$ be a finite set of root
  derivations. Then the Lie algebra generated by $\DD$ over $\KK$ is
  finite dimensional if and only if every cyclic subset of $\DD$ is
  simple.
\end{theorem*}

As an application, we extend to the non-complete case and to
arbitrary complexity previous results in \cite{De} and \cite{AHHL}. We
say that a linear algebraic group $G$ acts on a $\TT$-variety $X$
vertically if the action is effective, the image of $G$ in $\Aut(X)$ is normalized by $\TT$,
and $\KK(X)^\TT\subset\KK(X)^G$. In 
Theorem~\ref{demazure} we show that if a linear algebraic group $G$
acts on a $\TT$-variety $X$ vertically, then $G$ is \emph{a group of
  type~A}, i.e., a maximal semisimple subgroup of $G$ is
isomorphic to a factor group of the direct product
${\SL_{r_1}(\KK)\times\ldots\times\SL_{r_s}(\KK)}$ for some positive
integers $r_1,\ldots,r_s$ by a finite central subgroup.

\medskip

The content of the paper is as follows. In Section~\ref{comb-des} we
introduce the combinatorial description of $\TT$-varieties due to
Altmann, Hausen and S\"u\ss\ that we use in the paper. In
Section~\ref{s2} we provide the~announced classification of vertical
additive group actions on semiaffine $\TT$-varieties.  In
Section~\ref{s4} we show that a~simple set $\DD$ such that every
cyclic subset is also simple generates a Lie algebra isomorphic to
$\sl_r(\KK)$ for some $r\in \ZZ_{\ge 2}$. 
The theorem stated above is proved in Section~\ref{s5}. 
Finally, in Section~\ref{s6} we prove the~application 
to linear algebraic groups acting on a $\TT$-variety vertically.

\subsection*{Acknowledgements}

The authors would like to warmly thank the anonymous referee for
useful comments and corrections.

\section{Combinatorial description of $\TT$-varieties}
\label{comb-des}

In this chapter we briefly recall the combinatorial description of
$\TT$-varieties  used in this paper. For more details, see \cite{Oda,Fu,CLS}
for toric varieties and \cite{AlHa06,AHS08,AIPSV} for general $\TT$-varieties.

\subsection{Toric varieties}

In this subsection $\KK$ is a field of characteristic zero not
necessarily algebraically closed. Let $M$ be a lattice of rank $n$ and
$N=\homo(M,\ZZ)$ be its dual lattice. We let
$M_{\QQ}=M\otimes_\ZZ\QQ$, $N_{\QQ}=N\otimes_\ZZ\QQ$, and
$\langle\cdot,\cdot\rangle:N_\QQ\times M_\QQ\rightarrow \QQ$ be the
corresponding duality that we also denote by
$\langle v,u\rangle=v(u)$. Let $\TT=\Spec\KK[M]$ be the algebraic
torus whose character lattice is $M$. The~torus $\TT$ is an algebraic
group isomorphic to $\GM^n$, where $\GM$ is the multiplicative group
of the base field~$\KK$.

Recall that a toric variety $X$ is a normal variety endowed with a
faithful regular action $\TT\times X\rightarrow X$ of the
algebraic torus $\TT$ having an open orbit. A fan $\Sigma\in N_\QQ$ is
a finite collection of strictly convex polyhedral cones such that
every face of $\sigma\in\Sigma$ is contained in $\Sigma$ and for all
$\sigma,\sigma'\in\Sigma$ the intersection $\sigma\cap\sigma'$ is a
face in both cones $\sigma$ and $\sigma'$. A toric variety
$X(\Sigma)$ is built from $\Sigma$ in the following way. For every
$\sigma\in\Sigma$, we define an affine toric variety
$X(\sigma)=\Spec \KK[\sigma^{\vee}\cap M]$, where
$\sigma^{\vee}\subseteq M_{\mathbb{Q}}$ is the dual cone of $\sigma$
and $\KK[\sigma^{\vee}\cap M]$ is the semigroup algebra of
$\sigma^{\vee}\cap M$, i.e.,
\[
\KK[\sigma^{\vee}\cap M]=\bigoplus_{u\in\sigma^{\vee}\cap
  M}\KK\cdot\chi^{u},\quad\mbox{with}\quad\chi^{0}=1,\mbox{ and
}\chi^{u}\cdot\chi^{u'}=\chi^{u+u'},\ \forall u,u'\in\sigma^{\vee}\cap
M\,.
\]
Furthermore, if $\tau\subseteq\sigma$ is a face of $\sigma$, then the
inclusion of algebras
$\KK[\sigma^{\vee}\cap M]\hookrightarrow \KK[\tau^{\vee}\cap M]$
induces a $\TT$-equivariant open embedding
$X(\tau)\hookrightarrow X(\sigma)$ of affine $\TT$-varieties.  The
toric variety $X(\Sigma)$ associated to the fan $\Sigma$ is then
defined as the variety obtained by gluing the family
$\{X(\sigma)\mid\sigma\in\Sigma\}$ along the open embeddings
$X(\sigma)\hookleftarrow X(\sigma\cap\sigma')\hookrightarrow
X(\sigma')$ for all $\sigma,\sigma'\in\Sigma$.

\smallskip

Following \cite{GoLa73}, an algebraic variety $X$ is called semiaffine
if the morphism $X\rightarrow \Spec \Gamma(X,\OO_X)$ induced by the
global sections functor is proper. If, moreover, this morphism is
projective, we say that $X$ is semiprojective. In both cases,
$\Gamma(X,\OO_X)$ is finitely generated and so $\Spec \Gamma(X,\OO_X)$
is an affine variety \cite[Corollary~3.6]{GoLa73}. For instance,
complete or affine varieties are semiaffine, while projective or
affine varieties are semiprojective. Furthermore, any blow-up of a
semiaffine (resp. semiprojective) variety is also semiaffine
(resp. semiprojective).

A toric variety $X(\Sigma)$ is semiprojective if and only if the fan
$\Sigma$ is the normal fan of a (non-necessarily bounded) polyhedron
in $M_\QQ$, see~\cite[Proposition~7.2.9]{CLS}. Furthermore,
$X(\Sigma)$ is semiaffine if and only if the support of the fan
$\supp(\Sigma)=\bigcup_{\sigma\in\Sigma}\sigma$ is a convex set,
see~\cite[Exercise~15.1.7]{CLS}.

\subsection{Affine $\TT$-varieties and p-divisors}

In the remaining of this section $\KK$ is an algebraically closed
field of characteristic zero. Let $\sigma$ be a pointed polyhedral
cone in $N_{\QQ}$. We define $\pol_{\sigma}(N_{\QQ})$ to be the set of
all polyhedra with tail cone $\sigma$. We also include the empty set
$\emptyset \in \pol_{\sigma}(N_{\QQ})$. The set $\pol_\sigma(N_\QQ)$
with Minkowski sum has the structure of abelian semigroup with
identity, where the addition rule for $\emptyset$ is defined as
$\emptyset+\Delta=\emptyset$ for all
$\Delta\in \pol_{\sigma}(N_{\QQ})$.

Let $Y$ be a normal projective variety. A polyhedral divisor on $Y$ is
a formal sum $\DDD=\sum_{Z}\DDD_Z\cdot Z$, where
$\DDD_Z\in\pol_{\sigma}(N_{\QQ})$, $Z\subseteq Y$ is a prime divisor
and $\DDD_Z=\sigma$ for all but finitely many $Z$. We say that
$\sigma$ is the tail cone of $\DDD$ denoted by $\tail\DDD$. For every
$y\in Y$ we define the slice at $y$ by $\DDD_y=\sum_{Z\ni y} \DDD_Z$.  The
locus of $\DDD$ is $\Loc\DDD=\{y\in Y\mid \DDD_y\neq \emptyset\}$. For
every $u\in\tail\DDD^\vee$ we can evaluate $\DDD$ in $u$ by letting
$\DDD(u)$ be the $\QQ$-divisor in $\Loc\DDD$ given by
$$\DDD(u)=\sum_{Z\subseteq \Loc(\DDD)} \min\langle\DDD_Z,u \rangle \cdot
Z\,,$$ where $Z\subseteq \Loc(\DDD)$ runs through all prime divisor in
$\Loc(\DDD)$.

Consider a rational Cartier divisor $D$ on a normal variety $Y$.
The divisor $D$ is semiample if it admits a basepoint-free multiple,
i.e. for some $n\in\ZZ_{>0}$ the sets  $Y_f:=Y\setminus\Supp(\did(f)+nD)$,
where $f\in\Gamma(Y,\OO(nD))$, cover $Y$. Further, $D$ is big if for some
$n\in\ZZ_{>0}$ there is a section $f\in\Gamma(Y,\OO(nD))$ with an affine
locus $Y_f$.

A polyhedral divisor $\DDD$ on $Y$ is called a p-divisor if $\Loc\DDD$
is semiprojective and
\begin{enumerate}[$(i)$]
\item for every $u\in\tail\DDD^\vee$ the evaluation $\DDD(u)$ is
  semiample; 
\item for every $u\in\relint(\tail\DDD^\vee)$ the evaluation $\DDD(u)$
  is big.
\end{enumerate}

Let us recall that the complexity of an effective algebraic torus action $\TT\times X\to X$ on an algebraic variety $X$ is the difference $\dim X-\dim\TT$. In particular, toric varieties are precisely normal algebraic varieties with a torus action of complexity zero.

\smallskip

The main classification result for affine $\TT$-varieties in \cite{AlHa06} is the following.

\begin{theorem} \label{AH} To any p-divisor $\DDD$ on a
normal projective variety $Y$ one can associate a normal affine
$\TT$-variety $X(\DDD)=\Spec A(\DDD)$ of dimension $\rank N+\dim Y$
and complexity $\dim Y$ given by
$$A(\DDD)=\bigoplus_{u\in\tail\DDD^\vee\cap M} A_u\chi^u,\quad
  \mbox{where}\quad A_u=\Gamma(\Loc\DDD,\OO(\DDD(u))\subseteq
  \KK(Y)\,.
$$
Conversely, any normal affine $\TT$-variety is equivariantly
isomorphic to $X(\DDD)$ for some p-divisor $\DDD$ on some normal
projective variety $Y$.
\end{theorem}

\subsection{Arbitrary $\TT$-varieties and divisorial fans}
\label{sec:T-complete}

We recall the main facts of the description of not necessarily affine
$\TT$-varieties given in \cite{AHS08} in terms of divisorial fans. To
describe non-affine $\TT$-varieties, we need to describe first
$\TT$-invariant open embeddings.

Let $\Delta$ be a $\sigma$-polyhedron in $N_\QQ$ and let
$e\in \sigma^\vee$. The face of $\Delta$ defined by $e$ is the set
$$\face(\Delta,e)=\left\{p\in \Delta\mid
  p(e)=\min\langle\Delta,e \rangle\right\}\,.$$

Let now $\DDD=\sum_Z \DDD_Z\cdot Z$ and $\DDD'=\sum_Z \DDD'_Z\cdot Z$
be two p-divisors on $Y$. We define the support of $\DDD$ as the union
of divisors $Z\in Y$ such that $\DDD_Z\neq \emptyset$ or
$\DDD_Z\neq \tail(\DDD)$. We say that $\DDD'$ is a face of $\DDD$ if
$\DDD_Z'\subseteq \DDD_Z$ holds for all prime divisors $Z\subseteq Y$,
and for any $y\in\Loc(\DDD')$ there are $e_y\in \sigma^\vee\cap M$ and
$D_y$ in the linear system $|\DDD(e_y)|$ such that
\begin{enumerate}[$(i)$]
\item $y\notin \supp(D_y)$;
\item $\DDD'_y=\face(\DDD_y,e_y)$;
\item $\face(\DDD'_y,e_y)=\face(\DDD_{\tilde{y}},e_y)$ for every $\tilde{y}\in
Y\setminus \supp(D_y)$.
\end{enumerate}

Let $\DDD'$ be a face of $\DDD$. Then $A(\DDD)\subseteq A(\DDD')$ and
this inclusion is the comorphism of a $\TT$-equivariant open embedding
$X(\DDD')\hookrightarrow X(\DDD)$ of affine $\TT$-varieties
\cite[Proposition~3.4 and Definition~5.1]{AHS08}. We define the
intersection of two p-divisors as the polyhedral divisor
$$\DDD\cap\DDD'=\sum_{Z\subseteq Y}(\DDD_Z\cap\DDD_Z') \cdot Z\,.$$

Let $Y$ be a normal projective variety. A divisorial fan $\SF$ on $Y$
is a finite collection of p-divisors such that for every two
$\DDD,\DDD'\in \SF$ the intersection $\DDD\cap \DDD'$ is a face of
each and belongs to $\SF$. The set of tail cones
$\tail\DDD$ over all p-divisors $\DDD\in\SF$ is a fan called the tail
fan $\tail\SF$ of the divisorial fan $\SF$.

For every divisorial fan $\SF$ on $Y$ we
can define a $\TT$-scheme $X(\SF)$ by gluing the family of affine
$\TT$-varieties $\{X(\DDD)\mid \DDD\in \SF\}$ along the
$\TT$-invariant affine open sets
$\{X(\DDD\cap \DDD')\mid \DDD,\DDD'\in \SF\}$. The main result in
\cite{AHS08} is the following.

\begin{theorem}
Every normal $\TT$-variety is equivariantly isomorphic to $X(\SF)$
for some divisorial fan $\SF$ on a normal projective variety $Y$.
\end{theorem}

Furthermore, in \cite[Section~7]{AHS08}, conditions for $X(\SF)$ to be
separated and/or complete are given by proving a $\TT$-equivariant
version of the valuative criterion for separatedness and properness. A
straightforward application of these results provides a criterion for
a morphism of $\TT$-varieties to be proper.

For the following section, we need a technical lemma that we introduce
below. Let $\SF$ be a divisorial fan on a normal projective variety
$Y$ and $\DDD$ be a p-divisor in $\SF$. The support $\supp\SF$ is
defined as the union of the supports of all $\DDD\in\SF$. We define
the following p-divisors with affine locus.
\begin{enumerate}[$(i)$]
\item For every pair $(Z,v)$, where $Z\subseteq Y$ is a prime divisor
  and $v\in \DDD_Z$ is a vertex such that $\OO(\DDD(u))|_Z$ is big for
  $u\in \relint\left(\cone(\DDD_Z-v)^\vee\right)$, we let
  $\DDD_{Z,v}=v\cdot Z +\emptyset\cdot (\supp\SF\setminus
  Z)+\emptyset\cdot (Y\setminus Y_Z)$, where $Y_Z$ is an affine open
  set of $Y\setminus (\supp\SF\setminus Z)$ such that
  $Z\cap Y_Z\neq \emptyset$.
\item For every ray $\rho\in\tail(\DDD)$ such that $\DDD(u)$ is big
  for $u\in \relint(\rho^\bot\cap \sigma^\vee)$, we let
  $\DDD_{\rho}=\emptyset\cdot \supp(\SF)+\emptyset\cdot (Y\setminus
  Y_0)$ with tail cone $\tail\DDD_{\rho}=\rho$, where $Y_0$ is an
  affine open set of $Y\setminus \supp\SF$.
\end{enumerate}

Let $\mathcal{C}$ be the union of all the p-divisors $\DDD_{Z,v}$ and $\DDD_\rho$
defined above for all $\DDD\in\SF$. By \cite[Proposition~3.13]{PeSu11}, these p-divisors are
in bijective correspondence with $\TT$-invariant prime Weil divisors in
$X(\SF)$. It is easy to deduce from the proof that every
$\DDD'\in \mathcal{C}$ is a face of some $\DDD\in\SF$.
In particular, we have $\TT$-equivariant open embeddings
$X(\DDD')\hookrightarrow X(\DDD)\hookrightarrow X(\SF)$.

\begin{lemma}\label{T-codim2}
The union $U$ of the family of open sets
$\{X(\DDD')\subseteq X(\SF)\mid \DDD'\in \mathcal{C}\}$ is a big
open set in $X(\SF)$, i.e., $X(\SF)\setminus U$ has codimension at
least two in $X(\SF)$.
\end{lemma}

\begin{proof}
This is a local statement, so we can assume $\SF=\DDD$.  By
\cite[Proposition~3.13]{PeSu11}, the p-divisors on
$\DDD'\in\mathcal{C}$ are in bijection with prime $\TT$-invariant
divisors on $X(\DDD)$ and the union $U$
of $\TT$-invariant open sets $X(\DDD')$ with $\DDD'\in\mathcal{C}$
intersects all $\TT$-invariants divisors. This yields that the
complement of the $\TT$-invariant open set $U$ contains no prime
Weil divisor and so $X(\DDD)\setminus U$ has codimension at least two.
\end{proof}

\section{Vertical additive group actions on normal
  $\TT$-varieties} \label{s2}

In this section we provide a classification of additive group actions
on semiaffine $\TT$-varieties. Such a~classification was known
before in the particular case of affine $\TT$-varieties. 

\subsection{Preliminaries on additive group actions}

Let $\GA$ be the additive group of the base field $\KK$ of
characteristic zero but not necessarily algebraically closed. Regular
additive group actions on an affine variety $X$ are classified by
certain derivations on its ring of regular functions $\KK[X]$. We
briefly introduce this well known correspondence in this section. For
more details, see \cite{F}.

A derivation of a $\KK$-algebra $A$ is a linear map
$\partiala:A\rightarrow A$ satisfying
the Leibniz rule, i.e., $\partiala(fg)=f\partiala(g)+g\partiala(f)$ for
every $f,g\in A$. For an affine variety $X$, a derivation of $X$ is
just a derivation of its structure ring $\KK[X]$.

A derivation
$\partiala:\KK[X]\rightarrow \KK[X]$ on an affine variety $X$ is called
a locally nilpotent derivation (LND) if for every $f\in \KK[X]$ there
exists $i\in\ZZ_{> 0}$ such that $\partiala^i(f)=0$, where
$\partiala^i$ denote the $i$-th composition of $\partiala$ with itself.

The main classifying result for additive group actions on affine
varieties is the following.

\begin{theorem}
  Let $X$ be an affine variety and $\partiala:\KK[X]\rightarrow
  \KK[X]$ be a derivation on $X$. If $\partiala$ is locally nilpotent
  then the comorphism $\alpha:\GA\times X \rightarrow X$ of the
  exponential map
  $$\KK[X]\rightarrow \KK[t]\otimes_\KK\KK[X]\quad\mbox{given by}\quad
  f\mapsto \sum_i\frac{t^i\partiala^i(f)}{i!}$$
  defines a $\GA$-action on $X$. Conversely, every $\GA$-action on $X$
  arises in this way.
\end{theorem}

Let now $X$ be an arbitrary variety and $\mathcal{F}$ be a sheaf of
$\KK$-algebras on $X$. A derivation of $\mathcal{F}$ is a map
$\partiala:\mathcal{F}\rightarrow\mathcal{F}$ such that that
$\Gamma(U,\partiala):\Gamma(U,\mathcal{F})\rightarrow
\Gamma(U,\mathcal{F})$ is a derivation for every open affine set
$U\subseteq X$. A derivation $\partiala:\OO_X\rightarrow\OO_X$ of the
structure sheaf of $X$ is simply called a derivation on $X$. If $X$ is
affine, it is easy to check that every derivation
$\partiala:\KK[X]\to\KK[X]$ defines a derivation of the structure
sheaf $\OO_X$.  By a well known construction coming from differential
geometry, derivations on $X$ correspond to vector fields on $X$. The
set of all derivations $D:\OO_X\rightarrow \OO_X$ is denoted by
$\Lie(X)$ and has a natural structure of Lie algebra with Lie bracket
given by the commutator $[D,D']=D\circ D'-D'\circ D$.

Let now $X$ be a variety and $\alpha:\GA\times X\rightarrow X$ be
a $\GA$-action on $X$. In this case we obtain a derivation
$\partiala:\OO_X\rightarrow \OO_X$ of the structure sheaf of $X$ from
the $\GA$-action in the following way. Let $U\subseteq X$ be an affine
open set. Hence $\alpha^{-1}(U)$ is an open set containing
$\{0\}\times U\subset \GA\times X$. We define $\partiala$ via
$$\Gamma(U,\partiala):\Gamma(U,\OO_X)\rightarrow\Gamma(U,\OO_X),\qquad
f\mapsto \left[\frac{d}{dt}(\alpha^*f)\right]_{t=0}\,.$$ By abuse of
notation, when the affine open set $U$ is clear from the context we
will denote $\Gamma(U,\partiala)$ simply by $\partiala$. The
derivation $\partiala$ is called the infinitesimal generator of the
$\GA$-action.

The following result taken from \cite[Corollary~2.2]{DuLi16} provides
a classification of $\GA$-actions on semiaffine varieties in terms of
derivations of the structure sheaf generalizing the one for affine
case.

\begin{theorem} \label{open-LND} %
  Let $X$ be a semiaffine variety and
  $\partiala:\mathcal{O}_{X}\rightarrow\mathcal{O}_{X}$ be a derivation
  of the structure sheaf. Then $\partiala$ defines a regular
  $\GA$-action $\alpha$ on $X$ if and only if there exists a nonempty
  affine open set $U\subseteq X$ such that
  $\Gamma(U,\partiala):\Gamma(U,\mathcal{O}_{X})
  \rightarrow\Gamma(U,\mathcal{O}_{X})$ is locally nilpotent.
\end{theorem}

A derivation $\partiala:\OO_X\rightarrow\OO_X$ corresponding to a
$\GA$-action on $X$ is called a \emph{regularly integrable derivation}.

Any derivation $\partiala:\OO_X\rightarrow \OO_X$ on a normal variety
$X$ can be extended to a derivation
$\partiala:\KK(X)\rightarrow \KK(X)$ by the Leibniz rule. Indeed,
taking any affine open set $U\subseteq X$, we have $\KK(X)=\KK(U)$ and
for $f,g\in \KK[U]$ we define
$$\partiala\left(\frac{f}{g}\right)=
\partiala\left(f\cdot\frac{1}{g}\right)=
\partiala(f)\cdot\frac{1}{g}-\frac{f}{g^2}\cdot \partiala(g)=
\frac{g\partiala(f)-f\partiala(g)}{g^2}\,.$$ %
In particular, $\partiala:\OO_X\rightarrow\OO_X$ is completely
determined by its action on any affine open set $U\subseteq X$ or by
its action on rational functions $\KK(X)$. Furthermore, a derivation
$\partiala:\KK(X)\rightarrow\KK(X)$ of the field of rational functions
of $X$ is called regular if it defines a derivation of $\OO_X$, i.e.,
if for every affine open set $U\subseteq X$ we have
$\partiala(\KK[U])\subseteq \KK[U]$.

\subsection{Compatible $\GA$-actions on $\TT$-varieties}
Let $M$ be a lattice of finite rank and $X$ be a $\TT$-variety, where
$\TT=\Spec\KK[M]$. Let also $\partiala:\OO_X\rightarrow \OO_X$ be a
derivation on $X$. We say that $\partiala$ is \emph{homogeneous} with
respect to $M$ if $\Gamma(U,\partiala)$ is homogeneous as a linear map
for every $\TT$-invariant open set $U\subseteq X$, i.e., if
$\Gamma(U,\partiala)$ sends homogeneous elements to homogeneous
elements. If $\partiala$ is nonzero, we define the~degree of
$\partiala$ by $\deg\partiala=\deg\partiala(g)-\deg g$ for any homogeneous
$g\notin \ker\partiala$ in any $\TT$-invariant affine open set $U$.

\begin{definition}\label{def-hom-ga}
Let $\TT\times X\rightarrow X$ be a $\TT$-action on $X$. In
analogy with the notions of root and root subgroup of an algebraic
group, we say that an action
$\alpha:\GA\times X\rightarrow X$ on $X$ is \emph{compatible} if the~normalizer 
of the image of $\GA$ in $\Aut(X)$ contains the image of $\TT$. If
$\partiala$ is the regularly integrable derivation corresponding to
$\alpha$, then $\alpha$ is compatible if and only if there exists a
character $\chi$ of $\TT$ such that for every $\TT$-invariant affine
open set $U\subseteq X$ we have
$$\gamma\circ\partiala\circ\gamma^{-1}=\chi(\gamma)\cdot\partiala,\quad\mbox{for
all}\quad \gamma\in\TT\,.$$
\end{definition}

In the next lemma we characterize regularly integrable vector
fields that give rise to compatible $\GA$-actions.

\begin{lemma}\label{root-hom}
  A $\GA$-action $\alpha$ is compatible if and only if the
  corresponding regularly integrable derivation
  $\partiala:\OO_X\rightarrow\OO_X$ is homogeneous.
\end{lemma}

\begin{proof}
  Let $U\subseteq X$ be a $\TT$-invariant affine open set and let
  $\chi^e$, $e\in M$, be the character in the definition of compatible
  $\GA$-action. We have
  $\partiala=\chi^{-e}(\gamma)\cdot\gamma\circ\partiala\circ\gamma^{-1}$,
  $\forall\gamma\in\TT$. Take a homogeneous element
  $f\in \KK[U]$ of degree $u'$ and let
  $\partiala(f)=g=\sum_{u\in M} g_u\chi^u$. Now a routine computation
  yields
  \begin{align*}
    \sum_{u\in M}g_u\chi^u&=\partiala(f)=\chi^{-e}(\gamma)\cdot\gamma\circ
                            \partiala\circ\gamma^{-1}(f) \\
                          &=\chi^{-e-u'}(\gamma) \cdot\sum_{u\in M}\chi^u(\gamma)\cdot
                            g_u\chi^u,\qquad\forall\gamma\in\TT\,.
    \end{align*}
    This equality holds for all $\gamma\in \TT$ if and only if $g_u=0$
    for $u\neq e+u'$. In particular, $\partiala$ is homogeneous.
\end{proof}

Remark that the above proof shows that the character $\chi$ in the
definition of compatible $\GA$-action equals $\chi^e$, where $e$ is
the degree of $\partiala$. We need the following refinement of
Theorem~\ref{open-LND} for compatible $\GA$-actions on
$\TT$-varieties.

\begin{proposition}\label{T-open-LND}
  Let $X$ be a semiaffine $\TT$-variety and
  $\partiala:\mathcal{O}_{X}\rightarrow\mathcal{O}_{X}$ be a
  derivation of the structure sheaf. Then $\partiala$
  defines a compatible regular $\GA$-action $\alpha$ on $X$ if and
  only if $\partiala$ is homogeneous and  there exists a nonempty $\TT$-invariant affine open set
  $U\subseteq X$ such that
$\Gamma(U,\partiala):\Gamma(U,\mathcal{O}_{X}) \rightarrow\Gamma(U,\mathcal{O}_{X})$ is locally nilpotent.
\end{proposition}

\begin{proof}
  Assume first that $\alpha$ is compatible and let $\chi$ 
  be the character in the definition of compatible $\GA$-action. By
  Theorem~\ref{open-LND}, there exists an affine open set
  $U\subseteq X$ where $\partiala$ is locally nilpotent. 
  We~will show that we can take this open set $U$ to be
  $\TT$-invariant. Indeed, if $\partiala$ defines a compatible
  $\GA$-action, then the $\TT$-action and the $\GA$-action span an
  action of the solvable group $G=\GA\rtimes\TT$ where the~homomorphism for the semidirect product 
  $\TT\rightarrow \Aut(\GA)$ is given by
  $\gamma\mapsto (t\mapsto \chi(\gamma)\cdot t)$. By \cite{Ro63},
  there is a $G$-invariant open subset $U\subseteq X$ where a geometric quotient
  $U\mapsto U/G$ exists. By \cite[Theorem~10]{Ro56}, restricting
  the open set $U$ we can assume $U\simeq Y \times U/G$, where $Y$ is a general $G$-orbit in $X$.
  Being a homogeneous space of a solvable group, $Y$ is an affine variety. 
  Finally,  we can further restrict $U$ so that $U/G$ is affine, proving that
  $U$ can indeed be taken as an affine open subset; see also
  \cite[Theorem~3]{Po16} for a direct argument on the existence of such affine
  open subset $U$. In particular, $\partiala$ is a homogeneous LND on the
  $\TT$-invariant affine open set $U$.

  For the converse, if there is a $\TT$-invariant affine open set $U$
  where $\partiala$ is an LND, then Theorem~\ref{open-LND} shows that
  $\partiala$ defines a $\GA$-action on $X$ and such action is compatible
  by Lemma~\ref{root-hom} since $\partiala$ is homogeneous.
\end{proof}

\begin{definition}
Let now $X$ be a $\TT$-variety and $\KK(X)^{\TT}$ be the field of
$\TT$-invariant rational functions. A~homogeneous derivation
$\partiala:\OO_X\rightarrow\OO_X$ is called \emph{vertical} if
$\partiala(\KK(X)^\TT)=0$, and \emph{horizontal} otherwise. In case $\partiala$
is the derivation associated to a $\GA$-action, the corresponding
$\GA$-action is called vertical or horizontal, respectively.
\end{definition}

In geometric terms, the condition $\partiala(\KK(X)^\TT)=0$ means that 
a general orbit of the $\GA$-action on $X$ is contained in a $\TT$-orbit closure.

\subsection{Additive group actions on toric varieties}

In this section, we recall the characterization of compatible regular
$\GA$-actions on a toric variety $X$ in terms of the corresponding
regularly integrable derivations. For the results in this subsection,
the base field $\KK$ is not necessarily algebraically closed. These
results were first given in \cite{De} in an implicit way. Demazure's
approach was generalized and simplified in \cite{Lie10,Li,DuLi16}.

Let $X=X(\Sigma)$ be a toric variety. Let $\rho\in N$ and $e\in M$.
We define the linear map $\partiala_{\rho,e}:\KK[M]\rightarrow \KK[M]$,
$\chi^{u}\mapsto \rho(u)\chi^{u+e}$. It is a routine verification that
$\partiala_{\rho,e}$ satisfies the Leibniz rule and thus it defines a
homogeneous derivation on $\KK[M]$. Moreover, every homogeneous
derivation on $\KK[M]$ is a multiple of $\partiala_{\rho,e}$ for some
$\rho\in N$ and some $e\in M$. Since $\partiala_{a\rho,e}=a\partiala_{\rho,e}$,
without loss of generality we may and will assume in the sequel that
$\rho$ is primitive, i.e.  $\rho/n\notin N$ for all $n>1$. Let $\Sigma(1)$ denote 
the set of primitive vectors on the rays of $\Sigma$. The next theorem taken from \cite[Proposition~3.8]{DuLi16} gives a
classification of compatible regular $\GA$-actions on $X(\Sigma)$.

\begin{theorem}\label{toric-regular}
  Let $X(\Sigma)$ be a semiaffine toric variety.  Then $\partiala$ is
  the derivation of a compatible regular action
  $\alpha:\GA\times X(\Sigma)\rightarrow X(\Sigma)$ 
  if and only if $\partiala$ is a scalar multiple of
  $\partiala_{\rho,e}$ for some $\rho\in\Sigma(1)$ and $e\in M$ such that
  $\rho(e)=-1$, and $\rho'(e)\geq 0$ for all
  $\rho'\in\Sigma(1)\setminus\{\rho\}$.
\end{theorem}

\subsection{Higher complexity $\TT$-varieties}

We now describe vertical $\GA$-actions on higher complexity
$\TT$-varieties. In this subsection the field $\KK$ needs to be
algebraically closed. First, we need the following technical lemma.

\begin{lemma} \label{codim2}
Let $X$ be a normal variety and $\partiala:\KK(X)\rightarrow \KK(X)$
be a $\KK$-derivation of the field of rational functions on
$X$. Assume that $\partiala$ is regular on a big open set, i.e.,
$\partiala$ restricts to a derivation
$\partiala_U:\OO_U\rightarrow \OO_U$ on an open set $U\subseteq X$ with complement
of codimension at least~2. Then $\partiala$ is regular on $X$.
\end{lemma}

\begin{proof}
The statement is local, so without loss
of generality, we assume $X$ is affine. Now, since $\partiala|_{U}$
is regular, we have $\partiala|_{U}:\OO_U\rightarrow \OO_U$. Hence,
we can take global sections obtaining
$\Gamma(U,\partiala):\Gamma(U,\OO_X)\rightarrow\Gamma(U,\OO_X)$.
Since $X$ is normal, we have $\Gamma(U,\OO_X)=\Gamma(X,\OO_X)$
\cite[Corollary~11.4]{Eis95}. This yields that $\partiala$ is regular.
\end{proof}

Let $Y$ be a normal projective variety and $X=X(\SF)$ where $\SF$ is
a divisorial fan on $Y$. Let $\alpha$ be a vertical $\GA$-action on
$X$ with corresponding regularly integrable derivation
$\partiala:\OO_X\rightarrow \OO_X$ and $K_0=\KK(X)^\TT=\KK(Y)$ be the
field of $\TT$-invariant rational functions. The normalization
$\overline{X}$ of the base extension $X\times_{\Spec\KK} \Spec K_0$ is a
toric variety over the field $K_0$ with fan $\Sigma=\tail\SF$.

Since the $\GA$-action is vertical, the field of $\GA$-invariant
functions $\KK(X)^{\GA}$ contains $K_0$. This yields that the
$\GA$-action lifts to $\overline{X}$ and since
$\KK(X)=K_0(\overline{X})$, we have that both, the $\GA$-action on $X$
and the $\GA$-action on $\overline{X}$, are given by the same
derivation $\partiala:\KK(X)\rightarrow\KK(X)$. Now
Theorem~\ref{toric-regular} yields
$\partiala=\phi\partiala_{\rho,e}:=\partiala_{\phi,\rho,e}$ where
$\phi\in K_0=\KK(Y)$, $\rho\in \Sigma(1)$, and $e\in M$.

Let $P(0)$ denote the set of vertices of a polyhedron $P$. 
To state our classification of regular vertical $\GA$-actions on
$X(\SF)$ we need to define the following $\QQ$-divisor in analogy with
\cite[Theorem~2.4]{Li}: 
$$D_e:=\sum_{Z\subseteq
  \Loc\SF}\min\{v(e)\mid v\in \DDD_Z(0), \DDD\in \SF\}\cdot Z\,,
$$
where $\Loc\SF$ is the union of the $\Loc\DDD$ for all $\DDD\in\SF$,
see also \cite[Theorem~1.7]{ArLi12}.

\begin{theorem}\label{ver-regular} %
Let $X(\SF)$ be a semiaffine $\TT$-variety and $\Sigma=\tail\SF$. Then a derivation $\partiala$ of the field $\KK(X(\SF))$ is the derivation of a vertical compatible regular action
$\alpha:\GA\times X(\SF)\rightarrow X(\SF)$ if and only if $\partiala=\partiala_{\phi,\rho,e}$, where $e\in M$, $\phi\in\Gamma(\Loc\SF,\OO(D_e))$, and $\rho\in\Sigma(1)$ such that $\rho(e)=-1$, and $\rho'(e)\geq 0$ for all $\rho'\in\Sigma(1)\setminus\{\rho\}$.
\end{theorem}

\begin{proof}
  Let $\partiala=\partiala_{\phi,\rho,e}$. The fact that $\partiala$
  is locally nilpotent on any affine open set
  $X(\DDD)\subseteq X(\SF)$, where $\DDD\in\SF$ is such that $\rho_e$
  is a ray in $\tail\DDD$ is proven in \cite[Theorem~2.4]{Li}, see
  also~\cite[Theorem~1.7]{ArLi12}. Hence, by
  Proposition~\ref{T-open-LND} we only need to prove that $\partiala$
  is regular. Moreover, by Lemma~\ref{codim2} it is enough to verify
  that $\partiala$ is regular on a big open set and thus by
  Lemma~\ref{T-codim2} it suffices to verify that $\partiala$
  restricts to a regular derivation on $X(\DDD)$ for all $\DDD$ in the
  family $\mathcal{C}$ of p-divisors therein.

  Let $Z\subset X(\SF)$ be a prime divisor and $v$ be a vertex in some
  polyhedral coefficient $\DDD'_Z$ of a p-divisor $\DDD'$ in $\SF$.
  $\SF$. Assume further that $\DDD=\DDD_{Z,v}$ is contained in the
  family~$\CF$. Recall that the tail cone $\tail\DDD$ is
  $\{0\}\in N_\QQ$ and for every $u\in M$, the element
  $f\chi^u\in A(\DDD)$ if and only if $\ord_Z(f)+v(u)\geq 0$. Hence
  for $\partiala$ to leave $A(\DDD)$ invariant we need that for all
  $f\chi^u\in A(\DDD)$ the element $\phi f\chi^{u+e}\in A(\DDD)$. This
  yields
\begin{align*}
  \ord_Z\phi+\ord_Z f+v(u+e)&= \big(\ord_Z\phi+v(e)\big)+\big(\ord_Z f+v(u)\big)\geq 0
\end{align*}
The last inequality holds for all $u\in M$ if and only if
$\ord_Z\phi+v(e)\geq 0$. Indeed, recall that the evaluations $\DDD(u)$
are all $\mathbb{Q}$-Cartier. Replacing $u$ by a multiple, we can
assume that $\DDD(u)$ is Cartier which implies that it is locally
principal. Hence, there exists exists $f\in \KK(Y)$ satisfying
$\ord_Z(f)+\lfloor v(u)\rfloor=0$. This yields that the rational
function $\phi$ belongs to $\Gamma(\Loc\SF,\OO(D_e))$. The remaining
conditions of the theorem follow from Theorem~\ref{toric-regular}
since $\partiala$ defines a regular $\GA$-action on the toric
$\KK(Y)$-variety $\overline{X}$ given by the fan $\Sigma$.
\end{proof}

\section{Lie algebras generated by root derivations}
\label{s4}

In this and next sections we study Lie algebras generated by a finite
collection of derivations of a field. The study is motivated by the
results of the previous section.  On the other hand, the objects we
are dealing with are elementary, and for convenience of the reader we
re-introduce all necessary notions and definitions.

Let $\KK$ be an algebraically closed field of characteristic zero, $Y$
be an algebraic variety over $\KK$, and $\KK\subseteq K_0:=\KK(Y)$ be
the corresponding field extension. Let $M$ be a lattice of rank $n$
and $K_0(M)$ be the quotient field of the group algebra $K_0[M]$.

Let us introduce a class of $K_0$-derivations of the field
$K_0(M)$. Let $N=\Hom(M,\ZZ)$ be the dual lattice and let
$N\times M\to\ZZ$, $(v,u)\mapsto\langle v,u\rangle=v(u)$ be the
corresponding duality pairing. Given a fan $\Sigma\subseteq N_{\QQ}$,
we say that a vector $e\in M$ is a \emph{Demazure root} of the fan
$\Sigma$ if there is ray $\rho\in\Sigma(1)$ such that
$\langle \rho,e\rangle=-1$ and $\langle \rho',e\rangle\ge 0$ for all
$\rho'\in\Sigma(1)$, $\rho'\ne\rho$. We say that the ray $\rho$ is
\emph{associated} with the root $e$.

Let us define a \emph{root derivation} $D_{\phi,\rho,e}$ of the field
$K_0(M)$, where $\phi\in K_0$ and $e$ is a Demazure root of the fan
$\Sigma$ with the associated ray $\rho$, by the formula
$$
D_{\phi,\rho,e}(\chi^u)=\phi\cdot\rho(u)\cdot\chi^{u+e}.
$$
Since the field extension $K_0\subseteq K_0(M)$ is generated by the
elements $\chi^u$, $u\in M$, and $\rho$ is a linear function on $M$,
this formula defines a $K_0$-derivation of the field $K_0(M)$.

\begin{remark}
By Theorem~\ref{ver-regular}, the derivation corresponding to a
vertical $\GA$-action on a semiaffine $\TT$-variety $X(\SF)$ is a
root derivation of the field $\KK(X)=\KK(Y)(M)$ with respect to the
tail fan $\Sigma$.
\end{remark}

\smallskip

Let $m\in\ZZ_{\ge 2}$ and $\DD=\{D_{\phi_i,\rho_i,e_i}\}_{i=1}^m$ be a
set of root derivations of the field $K_0(M)$. We let
$$
D_i=D_{\phi_i,\rho_i,e_i}, \quad \phi(D_i)=\phi_i, \quad \rho(D_i)=\rho_i, \quad \text{and} \quad e(D_i)=e_i.
$$
Let us further assume that the derivations $D_i$ are pairwise non-proportional over the field~$K_0$.

\begin{definition} \label{def:sets}
  A set $\DD$ is called \emph{cyclic} if
  $\langle \rho_{i+1},e_i\rangle>0$ for all $i=1,\ldots,m$, where we
  set $\rho_{m+1}:=\rho_1$. A set $\DD$ is \emph{almost simple} if
$$
\langle \rho,e_i\rangle=
\begin{cases}
-1,&\mbox{if } \rho=\rho_i;\\
\ 1,&\mbox{if } \rho=\rho_{i+1};\\
\ 0,&\mbox{if } \rho\in\Sigma(1)\setminus\{\rho_i,\rho_{i+1}\}.
\end{cases}
$$
An almost simple set $\DD$ is \emph{simple} if
$\prod_{i=1}^m\phi_i\in\KK$. A simple set is \emph{very simple} if any of
its cyclic subsets is simple as well.
\end{definition}

\begin{remark}
  Without loss of generality we may assume in the definition of a very
  simple set that for every cyclic subset the product of corresponding
  functions $\phi_i$ equals $1$. Indeed, we take a point $y_0\in Y$
  such that all functions $\phi_i$ are defined and not equal zero at
  $y_0$, and replace each function $\phi_i$ by
  $\frac{1}{\phi_i(y_0)}\phi_i$. This changes every derivation $D_i$
  by a scalar multiple.
\end{remark}

We say that a root $e$ associated with a ray $\rho$ is
\emph{elementary}, if $\langle \rho,e\rangle=-1$,
$\langle \rho',e\rangle=1$ for some $\rho'\in\Sigma(1)$, and
$\langle \rho'',e\rangle=0$ for all
$\rho''\in\Sigma(1)\setminus\{\rho,\rho'\}$. A root derivation $D$ is
called \emph{elementary} if $e(D)=e$ with some elementary root
$e$. Clearly, a cyclic set $\DD$ is almost simple if and only if $\DD$
consists of elementary derivations.

\medskip

Consider the Lie algebra $\Der_{K_0}(K_0(M))$ of all $K_0$-derivations
of the field $K_0(M)$.  Let $\LIE(\DD)$ be the Lie algebra over the
field $K_0$ generated by $\DD$ in $\Der_{K_0}(K_0(M))$. Further, let
$\Lie(\DD)$ be the Lie algebra over the field $\KK$ generated by $\DD$
in the Lie algebra $\Der(K_0(M))$ of all $\KK$-derivations of the
field $K_0(M)$. We denote by $r=r(\DD)$ the number of pairwise
distinct elements of the set $\{\rho_1,\ldots,\rho_m\}$.

\begin{proposition} \  \label{propsl}
\begin{enumerate}[(1)]  %
\item If the set $\DD$ is almost simple, then the Lie algebra
  $\LIE(\DD)$ is isomorphic to $\sl_{\, r}(K_0)$.
\item If the set $\DD$ is very simple, then the Lie algebra
  $\Lie(\DD)$ is isomorphic to $\sl_{\,r}(\KK)$.
\end{enumerate}
\end{proposition}

In order to prove this proposition, we need an auxiliary result from
graph theory proven below in Lemma~\ref{elem}. Let $K_0$ be a field
and $\Gamma_r$ be a complete oriented graph on the set of vertices
$V=\{1,\ldots,r\}$. Assume that we are given a function
$\phi\colon E\to K_0^{\times}$, where $E$ is the set of edges of
$\Gamma_r$. For any oriented path $P$ in $\Gamma_r$ we let $\phi(P)$
be the product of $\phi(e)$ over all edges $e$ in $P$. By a
\emph{marking} on $\Gamma$ we mean a function $\phi$ such that
$\phi(C)=1$ for every oriented cycle $C$ in $\Gamma_r$.

We say that a function $\psi\colon V\to K_0^{\times}$ is a
\emph{potential} of a function $\phi$ if
$\phi([jk])=\psi(k)\psi(j)^{-1}$ for every edge $[jk]\in E$. It is
easy to see that a function $\phi$ admits a potential if and only if
$\phi$ is a marking.

Let $E'$ be a subset in $E$ such that there exists an oriented cycle
in $\Gamma_r$ passing through all vertices in $\Gamma_r$ whose set of
edged is contained in $E'$. A function $\phi'\colon E'\to K_0^{\times}$
such that $\phi'(C')=1$ for every oriented cycle $C'$ on $E'$ is said
to be a \emph{partial marking} of $\Gamma_r$.

\begin{lemma} \label{elem} Every partial marking $\phi'$ of the graph
  $\Gamma_r$ can be extended to a marking $\phi$ of $\Gamma_r$.
\end{lemma}

\begin{proof}
  For every edge $u=[jk]\in E$ we can find an oriented path $P_1$ from
  $j$ to $k$ in $E'$ and we let $\phi(u)=\phi(P_1)$. This function is
  well defined. Indeed, let $P_2$ be another oriented path from $j$ to
  $k$ in $E'$ and $P_3$ be an oriented path from $k$ to $j$ in
  $E'$. Then
$$
\phi(P_1)=\phi(P_3)^{-1}=\phi(P_2).
$$

Let us check that $\phi$ is a marking. Let $C$ be an oriented cycle in
$\Gamma_r$. We replace every edge $[jk]$ of $C$ which is not contained
in $E'$ by an oriented path $P$ from $j$ to $k$ in $E'$. We obtain an
oriented cycle $C'$ with $\phi(C)=\phi(C')$ and all edges of $C'$ are
in $E'$.  By definition of a partial marking we have $\phi(C')=1$.
\end{proof}

\begin{proof}[Proof of Proposition~\ref{propsl}]
We begin with the proof of (1). For our purposes we may assume that all functions $\phi_i$ are constants $1$. Let $\kappa_1,\ldots,\kappa_r$ be all pairwise distinct elements in the set $\{\rho_1,\ldots,\rho_m\}$. Assume first that the collection $\kappa_1,\ldots,\kappa_r$ is linearly independent in $N_\QQ$. Let us extend the collection $\kappa_1,\ldots,\kappa_r$ to a basis $\beta=\{\kappa_1,\ldots,\kappa_r,\tau_{r+1}\ldots,\tau_n\}$ of a sublattice $N'$ of finite index in $N$. The dual lattice $M':=\Hom(N',\ZZ)$ contains $M$ as a sublattice of finite index. Every $D_i$ extends to a $K_0$-derivation of the field $K_0(M')$. Using our basis $\beta$ and its dual basis $\beta^*$, we identify $K_0(M')$ with the field of rational functions $K_0(x_1,\ldots,x_r,y_{r+1},\ldots,y_n)$ where $x_i=\chi^{\kappa^*_i}$ and $y_i=\chi^{\tau^*_i}$ with $\kappa^*_i$ and $\tau^*_i$ the respective basis element dual to $\kappa_i$ and $\tau_i$ in $\beta^*$. Under this identification, the derivation $D_s$ from $\DD$ coincides with the derivation $\partial_{ji}:=x_j\frac{\partial}{\partial x_i}$, where $\kappa_i=\rho_s$ and $\kappa_j=\rho_{s+1}$. This shows that the algebra $\LIE(\DD)$ acts on the subalgebra $K_0[x_1,\ldots,x_r]$ via the standard representation of the algebra $\sl_{\,r}(K_0)$. In particular, $\LIE(\DD)$ is isomorphic to $\sl_{\,r}(K_0)$.

Assume now that the collection $\kappa_1,\ldots,\kappa_r$ is linearly dependent in $N_\QQ$. Then let $\lambda_1\kappa_1+\ldots+\lambda_r\kappa_r=0$ be a linear combination with $\lambda_i\in\QQ$. Computing the value of all $e_i$ on this combination, we obtain $\lambda_1=\ldots=\lambda_r$. Hence, up to a constant factor, the only linear relation is $\kappa_1+\ldots+\kappa_r=0$.

In particular, $\kappa_1,\ldots,\kappa_{r-1}$ are linearly independent. Let us extend the collection $\kappa_1,\ldots,\kappa_{r-1}$ to a basis $\beta=\{\kappa_1,\ldots,\kappa_{r-1},\tau_r\ldots,\tau_n\}$ of
a sublattice $N'$ of finite index in $N$. The dual lattice $M':=\Hom(N',\ZZ)$ contains $M$ as a sublattice of finite index. Every $D_i$ extends to a $K_0$-derivation of the field $K_0(M')$. Using our basis $\beta$ and its dual basis $\beta^*$, we identify $K_0(M')$ with the field of rational functions $K_0(x_1,\ldots,x_{r-1},y_{r+1},\ldots,y_{n+1})$ where $x_i=\chi^{\kappa^*_i}$ and $y_i=\chi^{\tau^*_i}$ with $\kappa^*_i$ and $\tau^*_i$ the respective basis element dual to $\kappa_i$ and $\tau_i$ in $\beta^*$. Let us embed the field $K_0(x_1,\ldots,x_{r-1},y_{r+1},\ldots,y_{n+1})$ into the field $K_0(z_1,\ldots,z_{r-1},z_r,y_{r+1},\ldots,y_{n+1})$ sending $x_i$ to $\frac{z_i}{z_r}$. Then the derivation $D_s$ from $\DD$ coincides with the restriction to $K_0(x_1,\ldots,x_{r-1},y_{r+1},\ldots,y_{n+1})$ of the derivation $z_j\frac{\partial}{\partial z_i}$
of the field $K_0(z_1,\ldots,z_r,y_{r+1},\ldots,y_{n+1})$, where $\kappa_i=\rho_s$ and $\kappa_j=\rho_{s+1}$. This shows that the algebra $\LIE(\DD)$ acts on the subalgebra $K_0[z_1,\ldots,z_r]$ via the standard representation of the algebra $\sl_{\,r}(K_0)$. In particular, $\LIE(\DD)$ is isomorphic to $\sl_{\,r}(K_0)$.

\medskip

Now we come to (2). This time we realize the $D_s$ as the derivation $\phi_s\partial_{ij}$ of field $K_0(x_1,\ldots,x_r)$. It~follows from the definition of an almost simple set that the indices $i,j$ determine the index $s$ uniquely. Let us denote the function $\phi_s$ as $\phi_{ij}$.

Consider a complete oriented graph $\Gamma_r$ on the set of vertices $V=\{1,\ldots,r\}$. Let $E'$ be the set of $m$ edges of $\Gamma_r$ corresponding to pairs $[ij]$ defined by the derivations $D_1,\ldots,D_m$. Consider a function
$$
\phi'\colon E'\to K_0^{\times}, \quad \phi'([ij])=\phi_{ij}.
$$
By definition of a very simple set, this function is a partial marking on the graph $\Gamma_r$.  By Lemma~\ref{elem}, such a partial marking can be extended to a marking $\phi\colon E\to K_0^{\times}$. In turn, this marking admits a potential $\psi\colon V\to K_0^{\times}$.

This shows that the derivations $\phi_{ij}\partial_{ij}$ are sent to the derivations $\partial_{ij}$ via an automorphism of the field $K_0(x_1,\ldots,x_r)$ given by $x_l\mapsto\psi(l)x_l$, $l=1,\ldots,r$. Thus the algebra $\Lie(\DD)$ is isomorphic to the $\KK$-subalgebra of $\Der(K_0(x_1,\ldots,x_r))$ generated by $\partial_{ij}$, $1\le i\ne j\le r$.  The latter algebra is obviously isomorphic to $\sl_{\,r}(\KK)$. This completes the proof of Proposition~\ref{propsl}.
\end{proof}

\begin{example}
Consider the fan $\Sigma$ in the lattice $\ZZ^2$ consisting of three two-dimensional cones generated by primitive vectors $\kappa_1=(1,0)$, $\kappa_2=(0,1)$, $\kappa_3=(-1,-1)$ with all their faces. We define a very simple set $\DD=\{D_1,\ldots,D_6\}$ by letting $\phi_1=\ldots=\phi_6=1$,
$$
e_1=(-1,1), \quad e_2=(0,-1), \quad e_3=(1,0), \quad e_4=(-1,0), \quad e_5=(0,1), \quad e_6=(1,-1),
$$
and
$$
\rho_1=\kappa_1, \quad \rho_2=\kappa_2, \quad \rho_3=\kappa_3, \quad \rho_4=\kappa_1, \quad \rho_5=\kappa_3, \quad\rho_6=\kappa_2.
$$
The corresponding derivations of the field $K_0(x_1,x_2)$ have the form
$$
x_2\frac{\partial}{\partial x_1}, \quad \frac{\partial}{\partial x_2}, \quad -x_1^2\frac{\partial}{\partial x_1}-x_1x_2\frac{\partial}{\partial x_2}, \quad \frac{\partial}{\partial x_1}, \quad -x_1x_2\frac{\partial}{\partial x_1}-x_2^2\frac{\partial}{\partial x_2}, \quad x_1\frac{\partial}{\partial x_2}.
$$
If we embed the field $K_0(x_1,x_2)$ into the field $K_0(z_1,z_2,z_3)$ by sending $x_i$ to $\frac{z_i}{z_3}$, $i=1,2$, these derivations coincide with restrictions of the derivations
$$
z_2\frac{\partial}{\partial z_1},\quad z_3\frac{\partial}{\partial z_2},\quad z_1\frac{\partial}{\partial z_3},\quad z_3\frac{\partial}{\partial z_1},\quad z_2\frac{\partial}{\partial z_3},\quad z_1\frac{\partial}{\partial z_2}.
$$
The algebra $\LIE(\DD)$ is isomorphic to $\sl_{\,3}(K_0)$. Moreover, we have constructed
a representation of the Lie algebra $\sl_{\,3}(K_0)$ into $\Der(K_0[x_1,x_2])$.

\end{example}

%%%%%%%%%%%%%%%%%%%%%%%%%%%%%%%%%%%%%%%%%%%%%%%%%%%%%%%%%%%%%%%%%%%%%%%%%%%%%%%%%%%%%%

\section{A criterion for finite-dimensionality} \label{s5}

In this section we give a criterion for a Lie algebra generated by a
finite set of root derivations to be finite dimensional.

\begin{theorem} \label{findim} %
  Let $\DD=\{D_{\phi_i,\rho_i,e_i}\}_{i=1}^m$ be a set of root
  derivations of the field $K_0(M)$.
  \begin{enumerate}[(1)]
  \item The Lie algebra $\LIE(\DD)$ is finite dimensional if and
    only if every cyclic subset of $\DD$ is almost simple.
  \item The Lie algebra $\Lie(\DD)$ is finite dimensional if and only
    if every cyclic subset of $\DD$ is simple.
\end{enumerate}
\end{theorem}

In the remaining of this section we prove Theorem~\ref{findim}
in several steps. We begin with a commutator formula. The proof
of the following lemma is straightforward.

\begin{lemma} \label{lemma} %
  Let $D_i=D_{\phi_i,\rho_i,e_i}$ and $D_j=D_{\phi_j,\rho_j,e_j}$ be
  root derivations. Then the commutator is given by
   $$[D_i,D_j](\chi^u)=\phi_i\phi_j(\langle \rho_i,e_j\rangle\langle
   \rho_j,u\rangle - \langle \rho_i,u\rangle\langle
   \rho_j,e_i\rangle)\chi^{u+e_i+e_j}\quad\mbox{for all}\quad u\in M.$$
\end{lemma}

Let us start with the ``only if'' direction in (1). We will prove that
if $\DD$ contains a cyclic subset that is not almost simple then
$\LIE(\DD)$ is infinite dimensional.  Without loss of generality we
may assume that
\begin{enumerate}[(a)]
\item the set $\DD$ is cyclic but not almost simple;
\item the set $\DD$ contains no proper cyclic subset.
\end{enumerate}

Condition (b) implies that $\langle \rho_i,e_j\rangle=0$ for all
$j\ne i-1,i$. In particular, the rays $\rho(D), D\in\DD$, are pairwise
distinct. Let us denote these rays by
$\rho_1,\ldots,\rho_m$. Moreover, we have
$\langle\rho_i,\sum_{j=1}^m e_j\rangle=\langle
\rho_i,e_{i-1}\rangle-1$. Hence, under condition (b) we have that a
set
\begin{align}\label{ex-remark}
\DD \mbox{ is almost simple if and only if }\DD\mbox{ is cyclic and }
\langle\rho_i,\sum_{j=1}^m e_j\rangle=0\mbox{, for all }
  i\in\{1,\ldots,m\}\,.
\end{align}

We proceed by induction on $m$. Let $m=2$ and
$\DD=\{D_1,D_2\}$. We put $a=\langle \rho_2,e_1\rangle$. Since $\DD$ is
not almost simple, we may assume up to renumbering that either
$a\ge 2$ or $a=1=\langle \rho_1,e_2\rangle$ and $\langle \rho_3,e_2\rangle>0$ for some
$\rho_3\in\Sigma(1)$, $\rho_3\ne\rho_1$. Consider the derivation
\begin{align*}
		D_2' = \underbrace{[\dots [[D_1, D_2],D_2]\dots D_2]}_{a + 1}.
\end{align*}

\smallskip

{\it Claim 1.}\ The derivation $D_2'$ is non-zero.
\begin{proof}
Take a vector $u\in M$ with $\langle \rho_2,u\rangle=1$. We claim that
$D_2'(\chi^u)\ne 0$. Note that $D_2^2(\chi^u)=0$ and we have
\begin{align*}
(-1)^{a+1}D_2'(\chi^u)&=D_2^{a+1}D_1(\chi^u)-(a+1)D_2^aD_1D_2(\chi^u) \\
              &=\phi_1\phi_2(\langle \rho_1,u\rangle\langle
                \rho_2,u+e_1\rangle-(a+1)\langle \rho_2,u\rangle\langle \rho_1,u+e_2\rangle D_2^a(\chi^{u+e_1+e_2}) \\
              &=\phi_1\phi_2(\langle \rho_1,u\rangle (1+a)-(a+1)(\langle
                \rho_1,u\rangle+\langle
                \rho_1,e_2\rangle))D_2^a(\chi^{u+e_1+e_2}) \\
              &=-\phi_1\phi_2(a+1)\langle \rho_1,e_2\rangle D_2^a(\chi^{u+e_1+e_2}).
\end{align*}
Furthermore, we have
\begin{align*}
D_2^a(\chi^{u+e_1+e_2})&=\phi_2^a\langle \rho_2,u+e_1+e_2\rangle \langle \rho_2,u+e_1+2e_2\rangle\ldots\langle \rho_2,u+e_1+ae_2\rangle
\chi^{u+e_1+(a+1)e_2} \\
&=\phi_2^a(1+a-1)(1+a-2)\ldots(1+a-a)\chi^{u+e_1+(a+1)e_2}\neq 0.
\end{align*}
This shows  $D_2'(\chi^u)\neq 0$.
\end{proof}

{\it Claim 2.}\ The derivation $D_2'$ is a root derivation associated
with the ray $\rho_2$.
\begin{proof}
The derivation $D_2'$ is homogeneous of degree $e_1+(a+1)e_2$. We have
$$
\langle\rho_1,e_1+(a+1)e_2\rangle=-1+(a+1)\langle\rho_1,e_2\rangle\ge\langle\rho_1,e_2\rangle, \quad
\langle\rho_2,e_1+(a+1)e_2\rangle=-1, \quad \text{and}
$$
$$
\langle\rho',e_1+(a+1)e_2\rangle=\langle\rho',e_1\rangle+(a+1)\langle\rho',e_2\rangle\ge\langle\rho',e_2\rangle
$$
for all $\rho'\in\Sigma(1)\setminus\{\rho_1,\rho_2\}$. Thus the vector $e_1+(a+1)e_2$ is a Demazure root of the fan $\Sigma$ associated with the ray $\rho_2$.

Take a vector $w\in M$ with $\langle\rho_2,w\rangle=0$. Then $D_2(\chi^w)=0$ and the condition $\langle\rho_2,w+e_1+ae_2\rangle=a-a=0$ implies $D_2(D_2^aD_1(\chi^w))=0$. We conclude that $D_2'(\chi^w)=0$ and thus
$$
D_2'(\chi^u)=\phi\langle\rho_2,u\rangle\chi^{u+e_1+(a+1)e_2}
$$
for all $u\in M$ with some $\phi\in K_0$.
\end{proof}

{\it Claim 3.}\ We have $\langle\overline{\rho},\deg D_2'\rangle>\langle\overline{\rho},\deg D_2\rangle$, where $\overline{\rho}:=\sum_{\rho\in\Sigma(1)}\rho$.
\begin{proof}
Take $\rho\in\Sigma(1)$. The inequality
$$
\langle\rho,e_1+(a+1)e_2\rangle=\langle\rho,e_1\rangle+a\langle\rho,e_2\rangle+\langle\rho,e_2\rangle\ge\langle\rho,e_2\rangle
$$
is strict for $\rho=\rho_1$ if $a>1$ and is strict for $\rho=\rho_3$ if $a=1$.
\end{proof}

{\it Claim 4.}\ The set $\DD':=\{D_1,D_2'\}$ is again cyclic, but not
almost simple.
\begin{proof}
For the first assertion, we have
$$
\langle\rho_1,e_1+(a+1)e_2\rangle=-1+(a+1)\langle\rho_1,e_2\rangle >0.
$$
If $a>1$ then $\langle \rho_2,e_1\rangle>1$ and $\DD'$ is not almost simple. If $a=1$ and
$\langle \rho_3,e_2\rangle>0$ then
$$
\langle \rho_3,e_1+(a+1)e_2\rangle>\langle \rho_3,e_2\rangle>0
$$
and again $\DD'$ is not almost simple.
\end{proof}

Repeating this procedure with the pair $\{D_1,D_2'\}$ and so on, we obtain infinitely many root derivations having different degrees by Claim 3. This proves that the Lie algebra $\LIE(\DD)$ has
infinite dimension.

\smallskip

Now assume that $m>2$. Replace the set $\DD=\{D_1,\ldots,D_m\}$ by the
set
$$
\widehat{\DD}=\{D_1,\ldots,D_{m-2},\widehat{D}_{m-1}\},
$$
where $\widehat{D}_{m-1}=[D_{m-1},D_m]$. We claim that $\widehat{\DD}$
is again cyclic and not almost simple. Indeed, we have
$\langle \rho_{m-1},e_{m-1}+e_m\rangle=\langle \rho_{m-1},e_{m-1}\rangle
=-1$.  Take $u\in M$ with $\langle \rho_m,u\rangle=0$ and
$\langle \rho_{m-1},u\rangle >0$. By Lemma~\ref{lemma}, we have
$$
\widehat{D}_{m-1}(\chi^u)=-\phi_1\phi_2\langle \rho_{m-1},u\rangle\langle
\rho_m,e_{m-1}\rangle\chi^{u+e_{m-1}+e_m}\ne 0.
$$
Moreover, the derivation $\widehat{D}_{m-1}$ annihilates all functions
$\chi^u$ with $\langle\rho_{m-1},u\rangle=0$, and thus it is a nonzero
root derivation of degree $e_{m-1}+e_m$. Since
$\langle \rho_1, e_{m-1}+e_m\rangle=\langle \rho_1, e_m\rangle>0$ we
have that $\widehat{\DD}$ is cyclic. By Lemma~\ref{lemma}, we obtain
that $\rho(\widehat{D}_{m-1})=\rho(D_{m})$ and since the sum of degrees of
derivations from $\DD$ and $\widehat{\DD}$ are equal, we conclude that
$\widehat{\DD}$ is not almost simple by \eqref{ex-remark}. By the
induction hypothesis, the Lie algebra generated by $\widehat{\DD}$ is
infinite dimensional. Thus the Lie algebra $\LIE(\DD)$ is infinite
dimensional as well.

\bigskip
	
We proceed with the ``only if'' direction in (2). Assume that
the algebra $\Lie(\DD)$ is finite dimensional. We already know that
every cyclic subset in $\DD$ is almost simple. Suppose that there is an
almost simple subset $\{D_1,\ldots,D_m\}$ in $\DD$ which is not
simple. Since all derivations $D_i$ are elementary, we can take a
finite extension $M\subseteq M'$ of lattices such that
$K_0(M')\cong K_0(x_1,\ldots,x_n)$ and every derivation $D_i$ acts on
this field as $\phi_i\partial_{ji}$, where
$\partial_{ji}=x_j\frac{\partial}{\partial x_i}$; see the proof of
Proposition~\ref{propsl}. Consider the derivation
$$
D_{1,1}:=[\ldots [[D_1, D_2], D_3] \dots D_m],D_1]=(\phi_1\ldots\phi_m)D_1.
$$
Similarly, the derivation $D_{1,s}$ defined recursively via
$D_{1,s}:=[\ldots [[D_{1,s-1}, D_2], D_3] \dots D_m], D_1]$ equals
$(\phi_1\ldots\phi_m)^sD_1$ for any positive integer $s$. Hence if the
function $\phi_1\ldots\phi_m$ is non-constant then the algebra $\Lie(\DD)$ has infinite dimension, a contradiction.

\bigskip

Now we come to the ``if'' direction in (1) and (2).

\begin{proposition} \label{propnil} %
Assume that the set $\DD$ contains no cyclic subset. Then the Lie algebras $\LIE(\DD)$ and $\Lie(\DD)$ are finite-dimensional and nilpotent.
\end{proposition}

\begin{proof} We subdivide the proof into several lemmas.

\begin{lemma} \label{lemcom} %
Let $D_i,D_j\in\DD$. Then the commutator $[D_i,D_j]$ is either zero or a root derivation. More precisely, if $\langle \rho_i,e_j\rangle>0$ then $[D_i,D_j]=D_{\psi,\rho_j,e_i+e_j}$, with $\psi=\langle\rho_i,e_j\rangle\phi_i\phi_j$.
\end{lemma}

\begin{proof}
By Lemma~\ref{lemma}, if either $\rho_i=\rho_j$ or we have $\langle \rho_i,e_j\rangle=\langle \rho_j,e_i\rangle=0$, then $[D_i,D_j]=~0$. If $\langle \rho_i,e_j\rangle>0$ then
$\langle \rho_j,e_i\rangle=0$; otherwise $\{D_i,D_j\}$ is a cyclic subset. Now the assertion follows from Lemma~\ref{lemma}.
\end{proof}

Let now $\kappa_1,\ldots,\kappa_r$ be pairwise distinct elements of the set $\{\rho_1,\ldots,\rho_m\}$. Then every $\rho_j$ coincides with a unique $\kappa_s$, and we set $\alpha(j)=s$. This defines a map $\alpha\colon\{1,\ldots,m\}\to\{1,\ldots,r\}$.

\begin{lemma} \label{lemord}
One can reorder the set $\{1,\ldots,r\}$ in such a way that the condition
$\langle \rho_i,e_j\rangle>0$ implies $\alpha(i)>\alpha(j)$.
\end{lemma}

\begin{proof}
If for every $\rho_i$ there is an element $e_j$ with $\langle \rho_i,e_j\rangle>0$, then we find a cyclic subset in $\DD$. Hence there is $\rho_i$ such that $\langle \rho_i,e_j\rangle\le 0$
for every $j$. Then we set $\rho_i=\kappa_1$ and proceed by induction on $r$.
\end{proof}

From now on we assume that the condition $\langle \rho_i,e_j\rangle>0$
implies $\alpha(i)>\alpha(j)$.

\begin{lemma} \label{lempom}
  Let $D_i$, $D_j$, $D_s$ be elements in $\DD$. If
  $\langle \rho_i,e_j\rangle>0$ and $\langle \rho_s,e_i+e_j\rangle>0$,
  then $\alpha(s)>\alpha(j)$.  In particular, the set
  $\DD\cup\{[D_i,D_j]\}$ contains no cyclic set.
\end{lemma}

\begin{proof}
  If $\langle \rho_s,e_j\rangle>0$ then $\alpha(s)>\alpha(j)$. If
  $\langle \rho_s,e_i\rangle>0$ then
  $\alpha(s)>\alpha(i)>\alpha(j)$. The second assertion follows from
  the first one and Lemma~\ref{lemcom}.
\end{proof}

With any vector $e\in M$ we associate a vector $v(e)\in\ZZ^r$,
$v(e):=(\langle\kappa_1,e\rangle,\ldots,\langle\kappa_r,e\rangle)$. Let us
say that a vector $v\in\ZZ^r$ is \emph{appropriate} if there is an
index $i$ such that $v_j=0$ for $j<i$, $v_i=-1$, and $v_j\ge 0$ for
$j>i$.

Clearly, the vector $v(e_i)$ is appropriate for every root $e_i$
corresponding to a derivation in $\DD$. Moreover, by
Lemma~\ref{lempom} every multiple commutator
$[\ldots,[D_{i_1},D_{i_2}],\ldots,D_{i_k}]$ with
$D_{i_1},D_{i_2},\ldots,D_{i_k}\in\DD$ is either zero or the vector
$v(e_{i_1}+e_{i_2}+\ldots+e_{i_k})$ is appropriate.

\smallskip

The proof of the following lemma is elementary and left to the reader.

\begin{lemma} \label{lemele} Consider a finite collection of
  appropriate vectors in $\ZZ^r$. Then only finitely many linear
  combinations of these vectors with non-negative integer coefficients
  are appropriate vectors.
\end{lemma}

Since
$v(e_{i_1}+e_{i_2}+\ldots+e_{i_k})=v(e_{i_1})+v(e_{i_2})+\ldots+v(e_{i_k})$,
we conclude from Lemma~\ref{lemele} that only finitely many multiple
commutators of the derivations $D_1,\ldots,D_m$ can be nonzero. So the
Lie algebras $\LIE(\DD)$ and $\Lie(\DD)$ are finite-dimensional and
nilpotent. This completes the proof of Proposition~\ref{propnil}.
\end{proof}

\begin{proposition} \label{propss} %
  Assume that every element of $\DD$ is contained in a cyclic subset
  and every cyclic subset is almost simple. Then
  $\DD=\DD_1\sqcup\ldots\sqcup\DD_s$, where $\DD_1,\ldots,\DD_s$ are
  maximal cyclic subsets. The Lie algebra $\LIE(\DD)$ is isomorphic to
  $\sl_{\, r_1}(K_0)\oplus\ldots\oplus\sl_{\,r_s}(K_0)$ for some
  positive integers $r_1,\ldots,r_s$. Moreover, if every cyclic subset
  of $\DD$ is simple then the Lie algebra $\Lie(\DD)$ is isomorphic to
  $\sl_{\, r_1}(\KK)\oplus\ldots\oplus\sl_{\,r_s}(\KK)$.
\end{proposition}

\begin{proof}
  Let $\DD_1$ and $\DD_2$ be two different maximal cyclic subsets of
  $\DD$.

\begin{lemma} \label{lemdis} We have
  $\langle \rho_i,e_j\rangle=\langle \rho_j,e_i\rangle=0$ for any
  $D_i\in\DD_1$, $D_j\in\DD_2$.  In particular, the sets
  $\{\rho(D), D\in\DD_1\}$ and $\{\rho(D), D\in\DD_2\}$ are disjoint.
\end{lemma}

\begin{proof}
  First assume that $\rho_i=\rho_j$. Putting the cyclic set $\DD_2$
  into the cyclic set $\DD_1$ just before the element $D_i$ and
  starting from $D_j$, we obtain a bigger cyclic set, a
  contradiction. This proves that the sets $\{\rho(D), D\in\DD_1\}$
  and $\{\rho(D), D\in\DD_2\}$ are disjoint.

  Assume that $\langle \rho_i,e_j\rangle>0$. Since the root $e_j$ is
  elementary, the vector $\rho_i$ is defined uniquely by this
  property. But $D_j$ is contained in the cyclic set $\DD_2$, and thus
  we should have $\rho_i\in \{\rho(D), D\in\DD_2\}$, a~contradiction.
\end{proof}

By Proposition~\ref{propsl}, we have
$\LIE(\DD_i)\cong\sl_{\,r_i}(K_0)$ and
$\Lie(\DD_i)\cong\sl_{\,r_i}(\KK)$, where $r_i$ is the number of
pairwise distinct elements in the set $\{\rho(D), D\in\DD_i\}$. By
Lemma~\ref{lemma} and Lemma~\ref{lemdis}, the elements from $\DD_i$
and $\DD_j$ with $i\ne j$ commute. Since the Lie algebra $\sl_{\, r}$
is simple, we obtain the assertions of Proposition~\ref{propss}.
\end{proof}

\begin{proposition} \label{semidirect} %
  Assume that every cyclic subset in $\DD$ is almost simple. Denote by
  $\DD'$ the subset of elements of $\DD$ that are contained in cyclic
  subsets. Then we have a semidirect product structure
  $$
  \LIE(\DD)\cong\LIE(\DD')\rightthreetimes \Nf,
  $$
  where $\Nf$ is a finite-dimensional nilpotent ideal in
  $\LIE(\DD)$. If every cyclic subset of $\DD$ is simple, we have
  $$
  \Lie(\DD)\cong\Lie(\DD')\rightthreetimes \nf,
  $$
  where $\nf$ is a finite-dimensional nilpotent ideal in $\Lie(\DD)$.
\end{proposition}

\begin{proof}

We begin with a computational lemma.

\begin{lemma}

Take $D_i\in\DD'$, $D_j\in\DD\setminus \DD'$ and let
$\widehat{D}=[D_i,D_j]$. Then either $\widehat{D}=0$ or 
$\widehat{D}$ is a root derivation and $\widehat{D}$ is not 
contained in a cyclic subset of $\DD\cup\{\widehat{D}\}$.

\end{lemma}

\begin{proof}

We subdivide the proof into four cases.

\smallskip

{\it Case 1.} If $\rho_i=\rho_j$ or
$\langle \rho_i,e_j\rangle=\langle \rho_j,e_i\rangle=0$, then
$\widehat{D}=0$.

\smallskip

{\it Case 2.} If $\langle \rho_i,e_j\rangle>0$ and
$\langle \rho_j,e_i\rangle>0$, then $\{D_i,D_j\}$ is a cyclic set, a
contradiction with the choice of $D_j$.

\smallskip

{\it Case 3.} Assume that $\langle \rho_i,e_j\rangle>0$ and
$\langle \rho_j,e_i\rangle=0$. By Lemma~5.2, we have
$\widehat{D}=D_{\psi,\rho_j,e_i+e_j}$. Assume that $\widehat{D}$ is
contained in a cyclic subset, say
$\ldots,D_k,\widehat{D},D_s\ldots$. Then $\langle \rho_j,e_k\rangle>0$
and ${\langle \rho_s,e_i+e_j\rangle>0}$. If
$\langle \rho_s,e_j\rangle>0$ then we replace $\widehat{D}$ by $D_j$
in the cyclic subset and obtain a cyclic subset, a~contradiction with the
choice of $D_j$. If $\langle \rho_s,e_j\rangle=0$ and
$\langle \rho_s,e_i\rangle>0$ then we replace $\widehat{D}$ by the~elements
$D_j,D_i$ in the cyclic subset and obtain a cyclic subset
containing $D_j$, again a contradiction.

\smallskip

{\it Case 4.} Assume that $\langle \rho_i,e_j\rangle=0$ and
$\langle \rho_j,e_i\rangle>0$. Since the root $e_i$ is elementary, we
have $\rho_j\in\rho(\DD')$.  By Lemma~5.2, we obtain
$\widehat{D}=D_{\psi,\rho_i,e_i+e_j}$. Assume that $\widehat{D}$ is
contained in a cyclic subset, say
$\ldots,D_k,\widehat{D},D_s\ldots$. Then $\langle \rho_i,e_k\rangle>0$
and $\langle \rho_s,e_i+e_j\rangle>0$. By assumptions of
Proposition~5.10, the~root $e_i$ is elementary. It implies
$\langle \rho_j,e_i\rangle=1$. If $\langle \rho_s,e_j\rangle>0$ then
$\ldots,D_k,D_i,D_j,D_s\ldots$ is a cyclic subset, a contradiction. If
$\langle \rho_s,e_j\rangle\le 0$ and $\langle \rho_s,e_i\rangle>0$ then
$\rho_s=\rho_j$ ($e_i$ is elementary) and
$$
\langle \rho_s,e_i+e_j\rangle=\langle \rho_j,e_i\rangle+\langle
\rho_j,e_j\rangle=1-1=0,
$$
a contradiction. This concludes the proof of Lemma~5.11.
\end{proof}

If $\widehat{D}\ne 0$, we add $\widehat{D}$ to $\DD$ and continue the
process. Let us prove that for any $e_j\in e(\DD\setminus\DD')$ the number of
roots of the form $e_j+\sum_{e\in e(\DD')} \lambda_e e$ with non-negative integer
coefficients $\lambda_e$ is finite.

Since such a linear expression may be not unique, we consider only
expressions, where the sum $\sum\lambda_e$ is minimal. Consider the
subset $E\subset e(\DD')$ consisting of all elementary roots $e$ such
that $\lambda_e>0$. Then the subset $E$ contains no cycle, because the
sum of roots in a cycle is zero.  We claim that the number of roots
$e_j+\sum_{e\in E} \lambda_e e$ with given $e_j$ and $E$ is finite.

Since $E$ contains no cycle, we can order the rays
$(\kappa_1,\ldots,\kappa_p)$ associated with roots in $E$ in such a
way that for any $e\in E$ the vector
$v(e)=(\langle\kappa_1,e\rangle,\ldots,\langle\kappa_p,e\rangle)$ is
appropriate. A slight generalization of Lemma~5.7 shows that only
finitely many vectors of the form
$v(e_j)+\sum_{e\in E} \lambda_e v(e)$ have all coordinates at least
$-1$. This implies the claim.

Thus the number of non-zero commutators is also finite. This shows
that the process stops in finitely many steps with a set $\DD''$. In
particular, the set $\DD''\setminus\DD'$ is stable under taking
commutators with elements of $\DD'$.
 
By Proposition~\ref{propnil}, the Lie algebra generated by
$\DD''\setminus\DD'$ is finite dimensional and nilpotent. Since its
generating set is stable under taking commutators with elements of
$\DD'$, the same holds for the whole algebra. We conclude that
$\LIE(\DD)\cong\LIE(\DD')\rightthreetimes\LIE(\DD''\setminus\DD')$ and
$\Lie(\DD)\cong\Lie(\DD')\rightthreetimes\Lie(\DD''\setminus\DD')$. This
concludes the proof of Proposition~\ref{semidirect}.
\end{proof}

Finally, Proposition~\ref{propss} and Proposition~\ref{semidirect}
prove the ``if'' direction in (1) and (2). This completes the proof of
Theorem~\ref{findim}.

%%%%%%%%%%%%%%%%%%%%%%%%%%%%%%%%%%%%%%%%%%%%%%%%%%%%%%%%%%%%%%%%%%%%%%%%%%%%%%%%%%%%%%%%%%%

\section{A Demazure type theorem} \label{s6}

\begin{definition}
An affine algebraic group $G$ over the ground field $\KK$ is said to be \emph{a group of type~A} if a~maximal semisimple subgroup of $G$ is isomorphic to a factor group of 
${\SL_{r_1}(\KK)\times\ldots\times\SL_{r_s}(\KK)}$ for some positive integers $r_1,\ldots,r_s$ by a finite central subgroup.
\end{definition}

\begin{definition}
Let $X$ be an algebraic variety with an action of an algebraic torus $\TT$. An effective action $G\times X\to X$ of an algebraic group $G$ is said to be \emph{vertical}, if the image of $G$ in the group $\Aut(X)$ is normalized by $\TT$ and general $G$-orbits on $X$ are contained in closures of $\TT$-orbits.
\end{definition}

The following theorem is a generalization of a classical result due to Demazure: the automorphism group of a complete toric variety is an affine algebraic group of type~A, see \cite[Proposition~3.3]{De}.

\begin{theorem} \label{demazure} %
Let an affine algebraic group $G$ admit a vertical action on a $\TT$-variety $X$. Then $G$ is a~group of type~A.
\end{theorem}

\begin{proof}
By \cite[Theorem~3]{Su} we can embed $X$ equivariantly into a completion $X'$ so that now $\TT$ and $G$ act on the variety $X'$. In particular, the variety $X'$ is semiaffine. Let us also replace the group $G$ by its subgroup $G'$ generated by all $\GA$-subgroups in $G$. It is well know that $G'$ is a semidirect product of a maximal semisimple subgroup of $G$ and the unipotent radical of $G$. The torus $\TT$ acts on $G$ by automorphisms and thus preserves the subgroup $G'$. The $\TT$-action on the tangent algebra $\gf$ of $G'$ is diagonalizable and the group $G'$ is generated by $\TT$-normalized $\GA$-subgroups. They correspond to vertical regularly integrable derivations of $X'$ and the Lie algebra $\gf$ is generated by such derivations. Now, Theorem~\ref{findim} and Propositions~\ref{propss} and~\ref{semidirect} complete the proof.
\end{proof}

Remark that taking $X$ to be a complete toric variety we recover Demazure's result since any
effective action $G\times X\to X$ normalized by the acting torus $\TT$ is vertical.

%%%%%%%%%%%%%%%%%%%%%%%%%%%%%%%%%%%%%%%%%%%%%%%%%%%%%%%%%%%%%%%%%%%

\end{document}